\documentclass[11pt,reqno]{amsart}
\usepackage{amsmath}
\usepackage{amsfonts}
\usepackage{amssymb}
\usepackage{mathrsfs}
\usepackage{cite}

\newcommand{\Aut}[1]{\mathop{\mathrm{Aut}}(#1)}
\newcommand{\inv}{^{-1}}

\newcommand{\wh}{\widehat}

\newcommand{\End}[1]{\mathop{\mathrm{End}}(#1)}
\newcommand{\Thmname}{Theorem}
\newcommand{\Propname}{Proposition}
\newcommand{\Lemmaname}{Lemma}
\newcommand{\Definitionname}{Definition}
\newcommand{\ilim}{\varprojlim}
\newtheorem{Thm}{\Thmname}
\newtheorem{Prop}[Thm]{\Propname}

{\theoremstyle{definition}
}
{\theoremstyle{remark}
\newtheorem{Rmk}[Thm]{Remark}}
\newtheorem{Cor}[Thm]{Corollary}
{\theoremstyle{remark}
}

\newtheorem*{Lemma*}{Lemma}

\title{On the endomorphism monoid of a profinite semigroup}

\author{Benjamin Steinberg}
\address{School of Mathematics and Statistics\\
Carleton University \\
1125 Colonel By Drive\\
Ottawa, Ontario  K1S 5B6 \\
Canada}
\thanks{The author was supported in part by NSERC}
\email{bsteinbg@math.carleton.ca}
\date{\today}

\begin{document}

\begin{abstract}
Necessary and sufficient conditions are given for the endomorphism monoid of a profinite semigroup to be profinite.  A similar result is established for the automorphism group.
\end{abstract}

\maketitle

\section{Introduction}

A classical result in profinite group theory says that if $G$ is a profinite group with a fundamental system of neighborhoods of the identity consisting of open characteristic subgroups, then the group $\Aut G$ of continuous automorphisms of $G$ is profinite with respect to the compact-open topology~\cite{RZbook}.  This applies in particular to finitely generated profinite groups.  Hunter proved that the monoid of continuous endomorphisms $\End S$ of a finitely generated profinite semigroup $S$ is profinite in the compact-open topology~\cite{Hunterendo}.  This result was later rediscovered by Almeida~\cite{Almeida:2005bshort}, who was unaware of Hunter's result. Almeida was the first to use to good effect that $\End S$ is profinite.     In this note I give necessary and sufficient conditions for $\Aut S$ and $\End S$ to be profinite for a profinite semigroup $S$.  This came out of trying to find an easier proof than Almeida's for the finitely generated case.  This led me unawares to exactly Hunter's proof, which I afterwards discovered via a Google search.  Like Hunter~\cite{Hunterendo} and Ribes and Zalesskii~\cite{RZbook}, I give an explicit description of $\End S$ and $\Aut S$ as inverse limits in the case they are profinite.  I also deduce the Hopfian property for $S$ in this case.  Recall that a topological semigroup $S$ is \emph{Hopfian} if each surjective continuous endomorphism of $S$ is an automorphism.

\section{The main result}
My approach, like that of Hunter~\cite{Hunterendo}, but unlike that of Almeida~\cite{Almeida:2005bshort} and Ribes and Zalesskii~\cite{RZbook}, relies on the uniform structure on a profinite semigroup and Ascoli's theorem.   Recall that a congruence $\rho$ on a profinite semigroup $S$ is called \emph{open} if it is an open subset of $S\times S$.  It is easy to see that open congruences are precisely the kernels of continuous surjections from $S$ to finite semigroups~\cite[Chapter 3]{qtheor}.  A congruence $\rho$ on $S$ is called \emph{fully invariant} if, for all continuous endomorphisms $f\colon S\to S$, one has $(x,y)\in \rho$ implies $(f(x),f(y))\in \rho$.  Equivalently, $\rho$ is fully invariant if and only if $\rho \subseteq (f\times f)\inv({\rho})$ for all $f\in \End S$.
In group theory, it is common to call a subgroup invariant under all automorphisms `characteristic.'  As I do not know of any terminology in vogue for the corresponding notion for congruences, it seems reasonable to define a congruence $\rho$ on $S$ to be \emph{characteristic} if $(x,y)\in \rho$ implies $(f(x),f(y))\in \rho$ for all continuous automorphisms $f$ of $S$.  Again, this amounts to $\rho\subseteq (f\times f)\inv(\rho)$ for all $f\in \Aut S$.  Clearly, any fully invariant congruence is characteristic.

Let $(X,\mathscr U)$ and $(Y,\mathscr V)$ be uniform spaces.  Recall that a family $\mathscr F$ of functions from $X$ to $Y$ is said to be \emph{uniformly equicontinuous} if, for all entourages $R\in \mathscr V$, one has $\bigcap_{f\in \mathscr F}(f\times f)\inv (R)\in \mathscr U$.  Of course, a uniformly equicontinuous family consists of uniformly continuous functions. It is clearly enough to have this condition satisfied for all $R$ running over a fundamental system of entourages for the uniformity $\mathscr V$.

Every compact Hausdorff space $X$ has a unique uniformity compatible with its topology, namely the collection of all neighborhoods (not necessarily open) of the diagonal in $X\times X$.  The following theorem is a special case of the Ascoli theorem for uniform spaces.

\begin{Thm}[Ascoli]\label{Ascoli}
Let $X,Y$ be compact Hausdorff spaces equipped with their unique uniform structures and let $\mathscr C(X,Y)$ be the space of continuous map from $X$ to $Y$ equipped with the compact-open topology.  Then, for a family $\mathscr F\subseteq \mathscr C(X,Y)$, the following are equivalent:
\begin{enumerate}
\item $\mathscr F$ is compact in the compact-open topology;
\item $\mathscr F$ is closed and (uniformly) equicontinuous.
\end{enumerate}
\end{Thm}

In the case of a profinite semigroup $S$ the uniform structure is given by taking the open congruences as a fundamental system of entourages.  The multiplication on $S$ is uniformly continuous and all continuous endomorphisms of $S$ are uniformly continuous.  The following result gives a sufficient condition for a profinite semigroup to have a fundamental system of entourages consisting of open fully invariant congruences.  The \emph{index} of an open congruence $\rho$ on a profinite semigroup $S$ is the cardinality of $S/\rho$.

\begin{Prop}\label{fullyinvariant}
Let $S$ be a profinite semigroup admitting only finitely many open congruences of index $n$ for each $n\geq 1$.  Then $S$ has a fundamental system of open fully invariant congruences.  This applies in particular if $S$ is finitely generated.
\end{Prop}
\begin{proof}
Let $\mathscr F_n$ be the set of open congruences on $S$ of index at most $n$ and let ${\rho_n}=\bigcap \mathscr F_n$; it is open because $\mathscr F_n$ is finite.  Clearly, the family $\{\rho_n\mid n\geq 1\}$ is a fundamental system of entourages for the uniformity.  I claim $\rho_n$ is fully invariant.  Indeed, let $f\in \End S$ and ${\sigma}\in \mathscr F_n$.  Then if $p\colon S\to S/{\sigma}$ is the quotient map, one has $(f\times f)\inv({\sigma})=\ker pf$ and hence is of index at most $n$.  Thus \[(f\times f)\inv ({\rho_n})=(f\times f)\inv \left(\bigcap \mathscr F_n\right)=\bigcap_{{\sigma}\in \mathscr F_n}(f\times f)\inv({\sigma})\supseteq \bigcap \mathscr F_n={\rho_n}\] as required.

The final statement follows since if $X$ is a finite generating set for $S$, then any congruence of index $n$ on $S$ is determined by its restriction to the finite set $X\times X$.
\end{proof}

\begin{Rmk}
There are non-finitely generated profinite groups that satisfy the hypothesis of Proposition~\ref{fullyinvariant}.  For example, one can take the direct product of all finite simple groups (one copy per isomorphism class); see~\cite[Exercise 4.4.5]{RZbook}.
\end{Rmk}

It is well known that, for any locally compact Hausdorff space $X$, the compact-open topology turns $\mathscr C(X,X)$ into a topological monoid with respect to the operation of composition.  The main result of this note is:

\begin{Thm}\label{main1}
Let $S$ be a profinite semigroup.  Then $\End S$ (respectively, $\Aut S$) is compact in the compact-open topology if and only if $S$ admits a fundamental system of open fully invariant (respectively, characteristic) congruences.  Moreover, if $\End S$ (respectively, $\Aut S$) is compact, then it is profinite and the compact-open topology coincides with the topology of pointwise convergence.
\end{Thm}
\begin{proof}
I just handle the case of $\End S$ as the corresponding result for $\Aut S$ is obtained by simply replacing the words `fully invariant' by `characteristic' and `endomorphism' by `automorphism'.

First observe that $\End S$ is closed in $\mathscr C(S,S)$.   Indeed, suppose that $f\colon S\to S$ is a continuous map that is not a homomorphism.  Then there are elements $s,t\in S$ such that $f(st)\neq f(s)f(t)$.  Choose disjoint open neighborhoods $U,V$ of $f(st)$ and $f(s)f(t)$ respectively.  By continuity of multiplication one can find open neighborhoods $W,W'$ of $f(s)$ and $f(t)$ so that $W\cdot W'\subseteq V$.  Then let $N$ be the set of all continuous functions $g\colon S\to S$ such that $g(st)\subseteq U$, $g(s)\subseteq W$ and $g(t)\subseteq W'$.  Then $f\in N$ and $N$ is open in the compact-open topology.  Clearly, if $g\in N$, then $g(s)g(t)\in W\cdot W'\subseteq V$ and $g(st)\in U$, whence $g(st)\neq g(s)g(t)$.  Thus $\End S$ is closed.

Assume that $\End S$ is compact.  By Ascoli's theorem, it is uniformly equicontinuous. Let $\rho$ be an open congruence on $S$.  Then uniform equicontinuity implies that \[\sigma = \bigcap_{f\in \End S} (f\times f)\inv (\rho)\] is an entourage of the uniformity on $S$. Evidentally, $\sigma$ is a congruence.  It must contain an open congruence by definition of the uniformity on $S$ and so $\sigma$ is an open congruence (the open congruences being a filter in the lattice of congruences on $S$).  Since the identity belongs to $\End S$, trivially $\sigma \subseteq \rho$.  It remains to observe that $\sigma$ is fully invariant.  Indeed, if $g\in \End S$, then \[(g\times g)\inv(\sigma) = \bigcap_{f\in \End S}(fg\times fg)\inv(\rho)\supseteq \bigcap_{h\in \End S} (h\times h)\inv (\rho)=\sigma\] establishing that $\sigma$ is fully invariant.  Thus $S$ has a fundamental system of open fully invariant congruences.

Conversely, suppose that $S$ has a fundamental system of open fully invariant congruences. Uniform equicontinuity follows because if $\rho$ is an open fully invariant congruence, then for any $f\in\End S$, one has $(f\times f)\inv (\rho)\supseteq \rho$ and hence $\bigcap_{f\in \End S}(f\times f)\inv(\rho)\supseteq \rho$.  Since the set of entourages is a filter, it follows that $\bigcap_{f\in \End S}(f\times f)\inv(\rho)$ is an entourage.  Because the open fully invariant congruences form a fundamental system of entourages for the uniformity on $S$, this shows that $\End S$ is uniformly equicontinuous.

Compactness of $\End S$ is now direct from Ascoli's theorem.  Let us equip $S^S$ with the topology of pointwise convergence. Since the compact-open topology is finer than the topology of point\-wise convergence, the natural inclusion $i\colon \End S\to S^S$ is continuous.  As $\End S$ and $S^S$ are compact Hausdorff, it follows that $i$ is a topological embedding and hence the compact-open topology on $\End S$ coincides with the topology of pointwise convergence.  Also $\End S$ is totally disconnected being a subspace of $S^S$.  Thus $\End S$ is profinite.
\end{proof}

In  light of Proposition~\ref{fullyinvariant},  Hunter's result for finitely generated profinite semigroups (and the corresponding well-known result for automorphism groups of finitely generated profinite groups) is immediate.

\begin{Cor}
If $S$ is a finitely generated profinite semigroup, then $\End S$ is a profinite monoid and $\Aut S$ is a profinite group in the compact-open topology, which coincides with the topology of pointwise convergence.
\end{Cor}

Theorem~\ref{main1} also implies the converse of~\cite[Proposition 4.4.3]{RZbook}: a profinite group $G$ has profinite automorphism group if and only if it has a fundamental system of neighborhoods of the identity consisting of open characteristic subgroups.

\begin{Rmk}
If $S$ is a profinite semigroup generated by a finite set $X$, then we have the composition of continuous maps $\End S\to S^S\to S^X$ where the last map is induced by restriction.  Moreover, this composition is injective.  Since $\End S$ is compact, it follows that $\End S$ is homeomorphic to the closed space of all maps $X\to S$ that extend to an endomorphism of $S$ equipped with the topology of pointwise convergence.  In the case $S$ is a relatively free profinite semigroup on $X$, we in fact have $\End S$ is homeomorphic to $S^X$.  Under this assumption, if $T$ is the abstract subsemigroup generated by $X$ (which is relatively free in some variety of semigroups), then it easily follows that $T^X$ is dense in $S^X$ and so $\End T$ is dense in $\End S$.
\end{Rmk}

A corollary is the well-known fact that finitely generated profinite semigroups are Hopfian.  In fact, there is the following stronger result.

\begin{Cor}\label{Hopfian}
Let $S$ be a profinite semigroup admitting a fundamental system of open fully invariant congruences, e.g., if $S$ is finitely generated.  Then $S$ is Hopfian.
\end{Cor}
\begin{proof}
Suppose that $f\colon S\to S$ is a surjective continuous endomorphism that is not an automorphism and let $f(x)=f(y)$ with $x\neq y\in S$.  Then there is an open fully invariant congruence $\rho$ so that $(x,y)\notin \rho$.  Since $\rho$ is fully invariant, there is an induced endomorphism $f'\colon S/{\rho}\to S/{\rho}$, which evidentally is surjective.  Thus $f'$ is an automorphism by finiteness.  But if $[x],[y]$ are the classes of $x,y$ respectively, then $f'([x])=f'([y])$ but $[x]\neq [y]$.  This contradiction shows that $S$ is Hopfian.
\end{proof}

\begin{Rmk}
In fact a more general result is true.  Let $X$ be a compact Hausdorff space and let $M$ be a compact monoid of continuous maps on $X$ with respect to the compact-open topology.  Then every surjective element of $M$ is invertible cf.~\cite{Akin}.  The proof goes like this.  First one shows that the surjective elements of $M$ form a closed subsemigroup $S$ (its complement is the union over all points $x\in X$ of the open sets $\mathscr N(X,X\setminus \{x\})$ of maps $f$ with $f(X)\subseteq X\setminus \{x\}$).  Clearly, the identity is the only idempotent of $S$.  But a compact Hausdorff monoid with a unique idempotent is a compact group so every element of $S$ is invertible.  Consequently, any compact Hausdorff semigroup whose endomorphism monoid is compact must be Hopfian.
\end{Rmk}

Not all profinite semigroups have a fundamental system of open fully invariant congruences.  For instance, if $S$ is the Cantor set $\{a,b\}^{\omega}$ equipped with the left zero multiplication, then $S$ is a profinite semigroup and every continuous map on $S$ is an endomorphism.  In particular, the shift map $\sigma$ that erases the first letter of an infinite word is a surjective continuous endomorphism, which is not an automorphism.  Thus $S$ is not Hopfian and so $S$ does not have a fundamental system of open fully invariant congruences by Corollary~\ref{Hopfian}.  As another example, let $F$ be a free profinite group on a countable set of generators $X=\{x_1,x_2,\ldots\}$ converging to $1$~\cite{RZbook}.  Let $\sigma\colon F\to F$ be the continuous endomorphism induced by the shift $x_1\mapsto 1$ and $x_i\mapsto x_{i-1}$ for $i\geq 2$.  Then $\sigma$ is surjective but not injective and so $\End F$ is not profinite.

As is the case for automorphism groups of profinite groups~\cite[Proposition~4.4.3]{RZbook}, $\End S$ can be explicitly realized as a projective limit of finite monoids given a fundamental system of open fully invariant congruences on $S$. For finitely generated profinite semigroups, this was observed by Hunter~\cite{Hunterendo}.  It was pointed out to me by Luis Ribes that the realizations as a projective limit in the above sources, and in a previous version of this note, are slightly wrong.  The statement and the proof in the next theorem are based on a modification suggested by him that appears in the second edition of~\cite{RZbook}.

\begin{Thm}\label{Thm9}
Let $S$ be a profinite semigroup and suppose that $\mathscr F$ is a fundamental system of entourages for $S$ consisting of open fully invariant congruences.  If $\rho\in \mathscr F$, then there is a natural continuous projection $r_{\rho}\colon \End S\to \End {S/\rho}$.  Let $\wh{\rho}$ be the corresponding open congruence on $\End S$.  Let $\wh{\mathscr F}=\{\wh{\rho}\mid \rho\in \mathscr F\}$.  Then 
\begin{equation}\label{realizeaslimit}
\End S\cong \ilim\nolimits_{\wh{\mathscr F}}\End S/\wh{\rho}.
\end{equation}
The analogous result holds for  $\Aut S$ if there exists a fundamental system of open characteristic congruences for $S$.
\end{Thm}
\begin{proof}
First we must show that $r_{\rho}$ is continuous so that $\wh{\rho}$ is indeed an open congruence.  Indeed, if $f\in \End S$, then $r_{\rho}\inv r_{\rho}(f)$ consists of those endomorphisms $g\in \End S$ that take each block $B$ of $\rho$ into the block of $\rho$ containing $f(B)$.  But since each block of $\rho$ is compact and open, and there are only finitely many blocks, it follows that $r_{\rho}\inv r_{\rho}(f)$ is an open set in the compact-open topology on $\End S$.  Thus $\wh{\rho}$ is an open congruence.

Since the open fully invariant congruences on $S$ are closed under intersection, the set $\wh{\mathscr F}$ is closed under intersection and so it makes sense to form the projective limit in \eqref{realizeaslimit}.  Since the canonical homomorphism from $\End S$ to the inverse limit on the right hand side of \eqref{realizeaslimit} is surjective, to prove that it is an isomorphism it suffices to show that $\wh{\mathscr F}$ separates points.  If $f,g$ are distinct endomorphisms of $S$, we can find $s\in S$ so that $f(s)\neq g(s)$.  Then since $\mathscr F$ is a fundamental system of entourages, there exists $\rho\in \mathscr F$ such that $(f(s),g(s))\notin \rho$.  It follows that $r_{\rho}(f)\neq r_{\rho}(g)$.
\end{proof}

\section*{Acknowledgments}
I would like to thank Luis Ribes for pointing out an error in the original version of Theorem~\ref{Thm9}.

\end{document}